\documentclass[reqno,final,a4paper]{amsart}
\usepackage{color}
\usepackage{amsmath, amssymb, amsthm, thmtools}
\usepackage[colorlinks=true,allcolors=blue
]{hyperref}
\usepackage{mathrsfs}
\usepackage{mathtools}
\usepackage[linesnumbered,ruled,vlined]{algorithm2e}

\SetCommentSty{mycommfont}
\SetKwInOut{Input}{input}\SetKwInOut{Output}{output}

\usepackage[noadjust]{cite}
\usepackage[noabbrev,capitalize,nameinlink]{cleveref}
\crefname{equation}{}{}
\usepackage{fullpage}
\usepackage{graphics}
\usepackage{tikzscale}
\usepackage{pifont}
\usepackage{tikz}
\usetikzlibrary{fit, backgrounds, shapes, calc}
\usepackage{pgfplots}
\pgfplotsset{compat=1.12}
\usepackage{bbm}
\usepackage[T1]{fontenc}
\usetikzlibrary{arrows.meta}

\usepackage{environ}
\usepackage{framed}
\usepackage{url}

\usepackage[noend]{algpseudocode}
\usepackage[labelfont=bf]{caption}
\usepackage{framed}
\usepackage[framemethod=tikz]{mdframed}
\usepackage{appendix}
\usepackage{graphicx}
\usepackage[textsize=tiny]{todonotes}
\usepackage{tcolorbox}
\usepackage{xcolor}
\usepackage[shortlabels]{enumitem}
\crefformat{enumi}{#2#1#3}
\crefrangeformat{enumi}{#3#1#4 to~#5#2#6}
\crefmultiformat{enumi}{#2#1#3}
{ and~#2#1#3}{, #2#1#3}{ and~#2#1#3}
\allowdisplaybreaks

\usepackage{ifthen}
\numberwithin{equation}{section}
\newtheorem{theorem}{Theorem}[section]
\newtheorem{proposition}[theorem]{Proposition}
\newtheorem{lemma}[theorem]{Lemma}
\newtheorem{claim}[theorem]{Claim}
\crefname{claim}{Claim}{Claims}

\newtheorem{corollary}[theorem]{Corollary}

\newtheorem*{question*}{Question}

\theoremstyle{definition}
\newtheorem{definition}[theorem]{Definition}
\newtheorem{problem}[theorem]{Problem}

\newtheorem*{definition*}{Definition}

\newtheorem*{fact*}{Fact}

\crefname{fact}{Fact}{Facts}

\theoremstyle{remark}

\newcommand{\R}{\mathbb{R}}
\newcommand{\N}{\mathbb{N}}

\renewcommand{\Pr}{\mathbb{P}}
\newcommand{\E}{\mathbb{E}}
\newcommand{\eps}{\varepsilon}

\newcommand{\Var}{\operatorname{Var}}

\newcommand{\Ber}{\operatorname{Ber}}

\newcommand{\Bin}{\operatorname{Bin}}

\newcommand{\SMA}{\operatorname{SMALL}}
\newcommand{\dist}{\operatorname{dist}}

\def\calB{\mathcal{B}}

\def\calD{\mathcal{D}}

\def\calH{\mathcal{H}}

\def\calP{\mathcal{P}}

\def\epsilon{\varepsilon}

\newenvironment{proofclaim}[1][Proof of claim]{\begin{proof}[#1]}{\end{proof}}

\usepackage{mleftright}
\mleftright

\title{On the hitting time of Hamiltonicity in bipartite Dirac graphs}

\author{Yiting Wang}\address{Institute of Science and Technology Austria,~Klosterneuburg,~3400,~Austria.}\email{yiting.wang@ist.ac.at}

\thanks{Yiting Wang is supported by the European Research Council (ERC), via grant agreement ``RANDSTRUCT'' No.\ 101076777.}

\date{\today}

\begin{document}

\begin{abstract}
Let $\eps\in (0,1/2]$ and let $G$ be a balanced bipartite graph on $2n$ vertices with minimum degree at least $(1/2 + \eps)n$. Then, whp, the hitting time for minimum degree 2 coincides with the hitting time for Hamiltonicity. This extends Bollob\'{a}s--Kohayakawa and gives a bipartite analogue of Johansson's theorem.
As an immediate corollary, we deduce a sharp threshold result for Hamiltonicity in such graphs. 
\end{abstract}
\maketitle

\section{Introduction}
Recall that a Hamilton cycle is a vertex-spanning cycle. 
Deciding whether a graph is Hamiltonian (i.e.\ contains a Hamilton cycle) is one of Karp's 21 NP-complete problems~\cite{Karp}. 
A central theme in graph theory is to find sufficient conditions for the existence of a Hamilton cycle. 
The classical theorem of this type is Dirac's theorem, which states that if $G$ is a graph on $n\geq 3$ vertices with minimum degree at least $n/2$, then $G$ is Hamiltonian.
A bipartite analogue of Dirac's theorem was proved by Moon and Moser~\cite{Moon-Moser}:
\begin{theorem}[Moon, Moser]\label{thm:Moon-Moser}
    Let $G$ be a balanced bipartite graph on $2n$ vertices. If $\delta(G) \geq (n+1)/2$, then $G$ is Hamiltonian.
\end{theorem}
Note that the bound is best possible, since one can take two vertex-disjoint copies of $K_{n/2,n/2}$. Such a graph is not connected, let alone Hamiltonian.

Both Dirac and the Moon-Moser theorems apply only to very dense graphs. 
What about sparser graphs? Ajtai, Koml\'os, Szemer\'edi~\cite{Ajtai-Komlos-Szemeredi} and independently Bollob\'{a}s~\cite{Bollobas} famously proved that the Erd\H{o}s-Renyi random graph $G(n,p)$ for $p\geq (\log n+ \log \log n+ \omega(1))/n$ is Hamiltonian whp. 
It is easy to see that such a graph has only $(1+o(1))\binom{n}{2}p = O(n\log n)$ edges whp, which is much sparser than the $\Theta(n^2)$ density required by the two theorems above.
In fact, they proved a stronger hitting-time theorem. To state it, we first introduce the relevant definitions.

Let $G$ be a graph on $n$ vertices. The \emph{random subgraph process} on $G$ is defined as follows: Let $(e_1,\dots, e_{|E(G)|})$ be a uniformly random ordering of the edges of $G$. Then, starting from the empty graph $G_0$, define $G_{i+1} = G_i \cup \{e_{i+1}\}$ for all $0\leq i< |E(G)|$. 
Given a monotone increasing graph property\footnote{A graph property is a set of graphs that are closed under graph isomorphism. Monotone increasing means closed under edge additions.}
$\calP$, let $\tau(\calP)$ be the minimum number $t\in \N$ such that $G_t\in \calP$. Note that $\tau(\calP)$ is a random variable.
Let $\calD_2$ denote the property of having minimum degree $2$ and $\calH$ denote the property of being Hamiltonian. Define $\tau_2 = \tau(\calD_2)$ and $\tau_\calH = \tau(\calH)$. Deterministically, $\tau_\calH\geq \tau_2$, since a graph with a vertex of degree at most $1$ cannot be Hamiltonian.

The hitting time theorem of~\cite{Ajtai-Komlos-Szemeredi, Bollobas} states that if $G= K_n$, then $\tau_2 = \tau_{\calH}$ whp.
Combining this result with a simple first and second moment computation for vertices of degree at most $1$ yields the $G(n,p)$ result stated above. In fact, the hitting-time result is, in this sense, the strongest result one can ask for because it pins down the exact obstruction to Hamiltonicity. 
It is natural to study whether the same result still holds with $K_n$ replaced by a sparser graph. Several such results are known.
Bollob\'{a}s and Kohayakawa~\cite{Bollobas-Kohayakawa} proved it for $G = K_{n,n}$. Johansson~\cite{Johansson} proved it for $\eps$-super-Dirac graphs (i.e.\ $\delta(G)\geq (1/2+\eps)n$). This result was also proved by Alon and Krivelevich~\cite{Alon-Krivelevich} using a different approach. There are also results for $(n,d,\lambda)$-graphs~\cite{Frieze-Krivelevich, Alon-Krivelevich, Chen-Chen-Im-Wang} and the $n$-dimensional hypercube~\cite{Condon-Espuny-Girao-Kuhn-Osthus}.

In this paper, we are interested in bipartite $\eps$-super-Dirac graphs. 
The main result is the following:
\begin{theorem}\label{thm:main}
    Let $\eps \in (0,1/2]$ and $G$ be a balanced bipartite graph with $2n$ vertices and $\delta(G) \geq (1/2 + \eps)n$. Then, whp, we have $\tau_2(G) = \tau_\calH(G)$.
\end{theorem}
This result simultaneously extends the theorem of Bollob\'{a}s and Kohayakawa~\cite{Bollobas-Kohayakawa} from complete bipartite graphs and Johansson's theorem~\cite{Johansson} from $\eps$-super-Dirac graphs to bipartite $\eps$-super-Dirac graphs.
One may wonder if the minimum degree condition in~\Cref{thm:Moon-Moser} suffices for a hitting time result. This is unfortunately not true. We give a sketch in the concluding remarks (\Cref{prop:counter-example}).

By a simple first and second moment computation for vertices of degree at most $1$, we deduce the following corollary:
\begin{corollary}\label{cor:threshold}
    Let $\eps > 0$ and $G$ be a balanced bipartite graph on $2n$ vertices with $\delta(G) \geq (1/2+\eps)n$. 
    If $p\geq (\log n + \log \log n+\omega(1))/\delta(G)$, then the random subgraph $G_p$ is Hamiltonian whp.
    Moreover,
    if $p\leq (\log n + \log \log n-\omega(1))/\delta(G)$ and $G$ has at least $\beta n$ vertices of degree $\delta(G)$ for some $\beta > 0$, then $G_p$ is non-Hamiltonian whp.
\end{corollary}
This corollary says that $\eps$-super-Dirac bipartite graphs remain Hamiltonian even after edge subsampling. This is related to the well-studied topic of robustness of graph properties. 
For a comprehensive introduction to this topic, see the excellent survey by Sudakov~\cite{Sudakov-survey}.
Note that the additional assumption that many vertices have degree $\delta(G)$ is needed. For instance, a complete bipartite graph with $(1-\eps)n/2$ edges incident to one vertex removed satisfies~\Cref{cor:threshold} with $p = (\log n + \log \log n\pm \omega(1))/n$ rather than $(\log n + \log \log n\pm \omega(1))/\delta(G)$.

\subsection{Proof overview} 
A powerful tool in the study of Hamiltonicity is the rotation-extension technique introduced by P\'osa~\cite{Posa}. 
Starting from a path $P$, a P\'osa rotation replaces $P$ with a different path $P'$ such that $V(P) = V(P')$ but the endpoints of $P$ and $P'$ are different.
By performing P\'osa rotations iteratively, one can find many pairs of endpoints for which there is a path supported on $V(P)$ between them.
The benefit is that if an endpoint either has an edge to a vertex outside the path or closes the path to a cycle, then one immediately obtains a longer path. This is an extension.

Repeating the rotation-extension steps will either find a Hamilton cycle or show that the graph contains a large bipartite hole (i.e.\ two linear-sized vertex sets with no edge between them), which can then be used to derive contradictions.

To implement the rotation-extension technique in the hitting-time setting, we follow an approach pioneered by Montgomery~\cite{Montgomery} (also used in~\cite{Alon-Krivelevich}). We first show that, whp, the hitting-time graph contains a sparse expander. Sparsity is crucial when we union bound over all such graphs.
Second, we show that, whp, every sparse expander has many pairs of edges (called booster pairs) whose addition increases the length of the longest path. The minimum degree condition is crucial for finding many booster pairs.
By adding booster pairs for at most $2n$ rounds, the graph becomes Hamiltonian. This is possible simply because the total number of edges added is at most $4n$, so even the final graph is still sparse enough for the application of the second step.

However, the above strategy has a critical flaw for bipartite graphs.
The rotation-extension technique relies on having many edges between two linear-sized sets. This is clearly not true for bipartite graphs, since there are no edges between two sets on the same side. Moreover, even if the sets are on distinct sides, the minimum degree condition is only $(1/2+\eps)n$. Since the bipartite graph has $2n$ vertices, this translates to $(1/4+\eps/2)n$ for an $n$-vertex graph
and is too weak to employ the non-bipartite result.
Thus, the proof needs a genuinely bipartite, parity-aware booster mechanism. 
We utilise an idea of Bollob\'as and Kohayakawa~\cite{Bollobas-Kohayakawa}. 
An almost $2$-factor is a spanning subgraph whose components are cycles, except possibly one path. 
Instead of building longer and longer paths, we build larger and larger almost $2$-factors.
The proof consists of finding a sparse expander (\Cref{lem:sparse-expander-almost-2factor-exists}) and showing that there are booster pairs, modified for almost $2$-factors, for any sparse expander (\Cref{lem:existence-booster-pair}). 

\section{Preliminaries}
\subsection{Notation}
Let $G$ be a graph. 
We write $e(G)$ for $|E(G)|$. For two disjoint sets of vertices $S,T\subseteq V(G)$, we write $E_G(S,T)$ for the set of edges with one end in $S$ and another in $T$.
For $X\subseteq V(G)$, let $\Gamma_G(X) = \{v\in V(G):|E_G(\{v\}, X)|\geq 1\}$ and $N_G(X) = \Gamma_G(X)\setminus X$.
For two vertices $u,v\in V(G)$, we denote their graph distance by $\dist_G(u,v)$.
We also write $N_G^t(X)$ for the set of vertices that are at distance at most $t$ from a vertex in $X$ and are not in $X$. 
For a vertex $v\in V(G)$, we denote its degree in $G$ by $\deg_G(v)$. The maximum degree of $G$ is denoted by $\Delta(G)$.

A sequence of events $(E_n)_{n\in \mathbb N}$ happens with high probability if $\lim_{n\rightarrow \infty}\Pr(E_n) = 1$. We abbreviate this as whp. We write $[n]$ for $\{1,\dots, n\}$. All logarithms are base $e$. We use the usual asymptotic notation. Let $f(n)$ and $g(n)$ be two functions. If $\lim_{n\rightarrow \infty} f(n)/g(n) = 0$, then $f(n) = o(g(n))$ and $g(n) = \omega(f(n))$. If there exists an absolute constant $C>0$ such that $|f(n)|\leq C|g(n)|$, then we write $f(n) = O(g(n))$ and $g(n) = \Omega(f(n))$. We ignore ceilings and floors for asymptotic quantities. Let $\mathbb N = \{1,2,\dots\}$.

A random variable $X$ \emph{stochastically dominates} another random variable $Y$ if for every $t\in \R$, we have $\Pr(X\geq t)\geq \Pr(Y\geq t)$.

\subsection{Asymptotic equivalence of random graph models}
Let $G$ be a graph on $n$ vertices and $p\in [0,1]$. We denote by $G_p$ the random subgraph of $G$ obtained by keeping each edge of $G$ with probability $p$ independently. Also, let $0\leq m\leq e(G)$. We denote by $G_m$ the random subgraph of $G$ obtained by choosing uniformly among all subgraphs of $G$ with $m$ edges. For monotone graph properties, we can relate $G_p$ and $G_m$ as follows:
\begin{proposition}[\cite{random-graph-book}]\label{prop:model-equiv}
    Let $\delta > 0$.
    Let $G$ be a graph on $n$ vertices and $\mathcal P$ be a monotone graph property. Let $0\leq m\leq (1-\delta) e(G)$ and $p = m/e(G)$. If $G_p\in \calP$ whp then $G_m\in \calP$ whp.
\end{proposition}
Note that the hitting time is inherently a notion about $G_m$. This lemma allows us to work with $G_p$, which is often more convenient.  

Here is a lemma that works for the converse direction:
\begin{proposition}[\cite{random-graph-book}]\label{prop:Gm-to-Gp} 
    Let $G$ be a graph on $n$ vertices, $p\in [0,1]$ and $\calP$ a graph property. If for every $m = e(G)\cdot p+ O(\sqrt{e(G)p(1-p})$, we have $G_m\in \calP$ whp, then $G_p\in \calP$ whp.
\end{proposition}

For monotone increasing properties, the more edges the graph has, the more likely the property is to hold. We note the following simple fact:
\begin{proposition}\label{prop:monotonicity-m}
    Let $G$ be a graph and $m_1\leq t\leq m_2$. 
    There exists a coupling between $G_{m_1}, G_t$ and $G_{m_2}$ such that with probability $1$, $G_{m_1}\subseteq G_t\subseteq G_{m_2}$. 
\end{proposition}
\begin{proof}
    For each edge $e\in E(G)$, assign an independent random variable $U_e$ that is uniformly distributed on $[0,1]$. Order them so that $U_{e_1}\leq U_{e_2}\leq \cdots \leq U_{e(G)}$. 
    Let $E_k = \{U_{e_1},\dots, U_{e_k}\}$. Observe that $G_{m_1}$ is distributed as a random subgraph of $G$ with $E_{m_1}$, $G_t$ with $E_{t}$ and $G_{m_2}$ with $E_{m_2}$. 
    Clearly, $G_{m_1}\subseteq G_t\subseteq G_{m_2}$ with probability $1$.
\end{proof}

\subsection{Basic estimates}
We will make frequent use of the following estimates:
\begin{proposition}\label{prop:estimates}
    Let $n,s,t\in \N$ and $p \in (0,1)$.
    \begin{enumerate}
        \item $\binom{n}{t}\leq (en/t)^t$.
        \item $1-p\leq \exp(-p)$.
        \item $\binom{n-s}{t}\leq \binom{n}{t}\cdot \exp(-st/n)$.
        \item $\binom{n-s}{t-s}\leq \binom{n}{t}(t/n)^s$.
    \end{enumerate}
\end{proposition}

\subsection{Probabilistic inequalities}
We will use the following simple fact:
\begin{proposition}\label{prop:basic-inequality}
    Let $X\sim \Bin(n,p)$. For any $t > 0$, we have $\Pr(X\geq t)\leq \binom{n}{t}p^t$. For any $t< np$, we have $\Pr(X\leq t)\leq t\binom{n}{t}p^t(1-p)^{n-t}$.
\end{proposition}
\begin{proof}
    For the former, let the $X_i$ be independent $\Ber(p)$ random variables such that $X = \sum_i X_i$.
    Note that $X\geq t$ implies that there exists $I\subseteq [n]$ such that $X_i = 1$ for all $i\in I$, which, by the union bound, happens with probability at most $\binom{n}{t}p^t$.
    
    For the latter, recall that for any $\ell \leq np$, the quantity $p_\ell = \Pr(\Bin(n,p) = \ell)$ is increasing in $\ell$. It follows that $\Pr(X\leq t)$ is at most $t$ times the term maximizing $p_\ell$ under $\ell\leq t\leq np$, namely $p_t = \binom{n}{t}p^t(1-p)^{n-t}$.
\end{proof}
Recall multiplicative Chernoff's bound:
\begin{proposition}\label{prop:chernoff}
    Let $X\sim \Bin(n,p)$. For any $\eps\in (0,1)$, we have 
    \begin{align*}
        &\Pr(X\geq (1+\eps)\E[X]) \leq \exp(-\eps^2\E[X]/3)\,,\\
        &\Pr(X\leq (1-\eps)\E[X]) \leq \exp(-\eps^2\E[X]/2)\,.
    \end{align*}
\end{proposition}
For very large deviations, here is a convenient form:
\begin{proposition}\label{prop:chernoff-large}
    Let $X\sim \Bin(n,p)$. Then, $\Pr(X\geq t)\leq 2^{-t}$ for $t\geq 2e\cdot \E[X]$.
\end{proposition}
Finally, we need a hypergeometric Janson inequality. See, for instance,~\cite[Lemma 3.1]{Balogh-Morris-Samotij-Warnke}.
\begin{proposition}\label{prop:janson}
    Suppose $\{B_i\}_{i\in I}$ is a family of subsets of an $n$-element set $\Omega$, let $m\in \{0,\dots, n\}$, and let $p = m/n$. Let $\mu = \sum_{i\in I}p^{|B_i|}$ and $\Delta = \sum_{i\sim j} p^{|B_i\cup B_j|}$, where the second sum is over all ordered pairs $(i,j)\in I^2$ such that $i\neq j$ and $B_i\cap B_j \neq \emptyset$. 
    Let $R$ be a uniformly chosen random $m$-set of $\Omega$, and let $\calB$ denote the event that $B_i\nsubseteq R$ for all $i\in I$. Then, for every $q\in [0,1]$, 
    \[
        \Pr(\calB)\leq 2\cdot \exp(-q\mu + q^2\Delta/2)\,.
    \]
\end{proposition}

\subsection{2-factors and expanders}\label{sec:2-factor}
Let $G$ be a connected graph.
A \emph{$2$-factor} $H\subseteq G$ is a spanning 2-regular subgraph of $G$. For a bipartite graph $G$, there are no odd cycles, and thus a $2$-factor is simply a vertex-disjoint union of cycles.
For technical reasons, we need a notion of almost $2$-factors.
An \emph{almost $2$-factor} is a spanning subgraph of $G$ whose components are cycles, except possibly one path.

We introduce the definition of a bipartite expander:
\begin{definition}
    Let $\alpha \in (0,1)$ and $G$ be a bipartite graph on $A\cup B$ and $H\subseteq G$. We say $H$ is a \emph{$(\alpha n, 2)$-expander} if for every $X\subseteq A$ with $|X|\leq \alpha n$, we have $|N_H(X)|\geq 2|X|$ and for every $Y\subseteq B$ with $|Y|\leq \alpha n$, we have $|N_H(Y)|\geq 2|Y|$.
\end{definition}
Let $G$ be a bipartite graph on $A\cup B$ and $S\subseteq A$. Let $\Gamma^{=1}(S) = \{v\in B: |E(v,S)| = 1\}$ and $\Gamma^{\geq 2}(S) = \{v\in B: |E(v,S)| \geq 2\}$. Similarly, when $S\subseteq B$, define $\Gamma^{=1}(S)$ and $\Gamma^{\geq 2}(S)$ with $B$ replaced by $A$.

The following is a necessary and sufficient condition for the containment of a $2$-factor in bipartite graphs (see e.g.~\cite[Lemma 5]{Bollobas-Kohayakawa}):
\begin{proposition}\label{prop:contain-2-factor}
    Let $G$ be a bipartite graph on $A\cup B$.
    Then 
    \begin{align*}
        \text{$G$ contains a $2$-factor} \quad\quad &\text{if and only if} \quad \text{for all $S\subseteq A$, we have $|\Gamma^{\geq 2}(S)| + \frac{1}{2}|\Gamma^{=1}(S)| \geq |S|$ }\\
        &\text{if and only if} \quad\text{for all $T\subseteq B$, we have $|\Gamma^{\geq 2}(T)| + \frac{1}{2}|\Gamma^{=1}(T)| \geq |T|$.}
    \end{align*}
\end{proposition}

\section{Proof}
Throughout the section, $G$ is a balanced bipartite graph on $2n$ vertices with $\delta(G) \geq (1/2+\eps)n$.
Let $\tau = \tau_2(G)$, $p_- = \log n/n$ and $p_+ = 2\log n/n$. 
Also, let $G^- = G_{p_-}$ and $G^+ = G_{p_+}$. Fix $\alpha = 1/10$.

\subsection{Estimating the range of 
\texorpdfstring{$\tau_2(G)$}{τ2(G)}}
We first control the range of $\tau$. The proof is a routine first and second moment computation. 
\begin{lemma}\label{lem:hitting-time-window}
    Whp $p_-\cdot e(G)\leq\tau\leq p_+ \cdot e(G)$.
\end{lemma}
\begin{proof}
    By~\Cref{prop:monotonicity-m,prop:model-equiv}, it suffices to show that 
    \begin{enumerate}[label=(\alph*)]
        \item whp $G^-$ has a vertex of degree at most $1$,
        \item whp $G^+$ has no vertex of degree at most $1$.
    \end{enumerate}
    
    For each $v\in V(G)$, let $X_v^-$ be the indicator random variable of the event $\deg_{G^-}(v)\leq 1$ and $X^- = \sum_{v}X_v$. Define $X_v^+$ and $X^+$ similarly for $G^+$. 

    For (a), let us restrict our attention to one side of the bipartite graph, say $A$. Note that $X_v^-$ stochastically dominates the indicator random variable of the event $\Bin(n, p_-)\leq 1$. Thus, 
    \begin{multline*}
        \E[X_A^-] = \sum_{v\in A}\E[X_v^-] \geq n\cdot \Pr(\Bin(n, \log n/n)\leq 1) 
        = n\cdot \left((1-\log n/n)^{n} + n\cdot \log n/n\cdot (1-\log n/n)^{n-1}\right)\\ = (1+o(1))\log n\rightarrow \infty\,.
    \end{multline*}
    Because the random variables $X_v^-$ are independent for different $v\in A$, we know $\Var(X_A^-)\leq \E[X_A^-]$. It then follows by Chebyshev's inequality that (a) holds whp.

    For (b), note that $X_v^+$ is stochastically dominated by the indicator of the event $\Bin((1/2 +\eps)n,p_+)\leq 1$. Thus,
    \begin{multline*}
        \E[X^+]\leq 2n\cdot \Pr(\Bin((1/2 +\eps)n, 2\log n/n)\leq 1) = 2n\cdot \Bigl((1-2\log n/n)^{(1/2 +\eps)n} \\+ (1/2 +\eps)n\cdot 2\log n/n\cdot  (1-2\log n/n)^{(1/2+\eps)n-1}\Bigr) = n^{-2\eps + o(1)}\,.
    \end{multline*}
    By Markov's inequality, whp (b) holds.
\end{proof}

Combining~\Cref{prop:model-equiv,prop:monotonicity-m,lem:hitting-time-window}, we prove a convenient lemma that allows the computations to be done in the $G_p$ model instead of the $G_m$ model:
\begin{corollary}\label{cor:monotonicity}
    Let $\calP$ be a monotone decreasing property. Then, if $G^+\in \calP$ whp, then so is $G_\tau$. Similarly, for a monotone increasing property $\calP$, if $G^-\in \calP$ whp, then so is $G_\tau$.
\end{corollary}

\subsection{Properties of \texorpdfstring{$G_\tau$}{Gτ}}
Choose $\xi = \xi(\eps) >0$ sufficiently small to satisfy various inequalities needed below.
Let $\SMA(H) = \{v\in V(H): d_{H}(v)\leq \xi \log n\}$. 
We now show that the graph $G_\tau$ satisfies the following four properties whp:

\begin{lemma}\label{lem:small-property}
        Whp, the following properties hold:
        \begin{enumerate}
            \item $\Delta(G_\tau)\leq 20\log n$,
            \item $|\SMA(G_\tau)|\leq n^{(1-\eps)/2}$,
            \item For every $u\in \SMA(G_\tau)$, there is no cycle of length at most $4$ containing $u$ in $G_\tau$,
            \item For distinct vertices $u,v\in \SMA(G_\tau)$, there is no path of length at most $4$ between them in $G_\tau$.
        \end{enumerate}
    \end{lemma}
    \begin{proof}        
        By~\Cref{cor:monotonicity}, for (1), it suffices to show that $\Delta(G^+)\leq 20\log n$ whp. 
       By~\Cref{prop:chernoff-large}, $\Pr(\deg_{G^+}(v)\geq 20\log n)\leq \Pr(\Bin(n, 2\log n/n)\geq 20\log n)\leq n^{-10}$. By the union bound, whp there is no vertex of degree at least $20\log n$ in $G^+$. 
        
        By~\Cref{cor:monotonicity}, for (2), it suffices to prove $|\SMA(G^-)|\leq n^{(1-\eps)/2}$.
        The degree of each vertex in $G^-$ is distributed as $\Bin(d_{G}(v), \log n/n)$. Thus, it stochastically dominates $\Bin((1/2+\eps)n, \log n/n)$. Applying~\Cref{prop:basic-inequality}, the probability that a vertex is in $\SMA(G^-)$ is at most 
        \begin{multline*}
            \Pr(\Bin((1/2+\eps)n, \log n/n)<\xi \log n)\leq \xi\log n\binom{(1/2+\eps)n}{\xi \log n}\left(\frac{\log n}{n}\right)^{\xi\log n}\left(1-\frac{\log n}{n}\right)^{(1/2+\eps)n - \xi\log n}\\
            \leq n^{-1/2 - \eps + \xi\log(e(1/2+\eps)/\xi)+o(1)}\leq n^{-1/2-3\eps/4}\,,
        \end{multline*}
        where the last inequality uses that $\xi$ is sufficiently small compared to $\eps$ so $\xi \log(e(1/2+\eps)/\xi) < \eps/4$.
        Thus, by Markov's inequality, whp we have $|\SMA(G^-)|\leq n^{(1-\eps)/2}$.

    For (3) and (4), we need to work with the $G_m$ model instead of the $G_p$ model.
    Let $m^- = e(G)\cdot p_-$ and $m^+ = e(G) \cdot p_+$. 
    By~\Cref{cor:monotonicity} and~\Cref{lem:hitting-time-window}, whp $m^-\leq \tau\leq m^+$.
    By~\Cref{prop:monotonicity-m}, there exists a coupling such that $G_{m^-}\subseteq G_\tau\subseteq G_{m^+}$ with probability $1$. 
    
    For (3), it suffices to show that no $u\in \SMA(G_{m^-})$ lies on a cycle of length at most $4$ in $G_{m^+}$. 
    Because the graph is bipartite, we only need to consider $4$-cycles.
    Fix an arbitrary vertex $v\in V(G)$.   
     The number of $4$-cycles containing $v$ is at most $\binom{n-1}{2}(n-3)$.
     Fix an arbitrary $4$-cycle containing $v$. 
     Using~\Cref{prop:estimates}(4),
     the probability that such a cycle exists in $G_{m^+}$ is at most 
     $$\frac{\binom{e(G)-4}{m^+ -4}}{\binom{e(G)}{m^+}}\leq \left(\frac{m^+}{e(G)}\right)^4 = (2\log n/n)^4\,.$$
     Let $d = \deg_G(v)$.
     Conditioned on the existence of a $4$-cycle in $G_{m^+}$, the probability that $v$ is in $\SMA(G_{m^-})$ is at most 
     \begin{equation}\label{eqn:computation}
         \begin{split}
             \sum_{i=0}^{\xi\log n}\binom{d}{i} \frac{\binom{e(G)-d+2}{m^- - i}}{\binom{e(G)}{m^-}} &= \sum_{i=0}^{\xi\log n}\binom{d}{i}\frac{\binom{(e(G)-i)-(d-2+i)}{m^- - i}}{\binom{e(G)}{m^-}} \\
         &\leq \sum_{i=0}^{\xi\log n}\left(\frac{ed}{i}\right)^i\cdot \frac{\binom{e(G) -i}{m^- - i}}{\binom{e(G)}{m^-}}\cdot \exp\left(-(1+o(1))d\cdot \frac{\log n}{n}\right)\\
         &\leq \sum_{i=0}^{\xi\log n}\left(\frac{ed}{i}\cdot \frac{\log n}{n}\right)^i n^{-1/2-\eps -o(1)}\\
         &\leq \sum_{i=0}^{\xi\log n} (e\log n/i)^i  n^{-1/2-\eps -o(1)}\leq n^{-1/2 - 3\eps/4}\,,
         \end{split}
     \end{equation}
     where we used~\Cref{prop:estimates}(4) in the second inequality,~\Cref{prop:estimates}(3) in the third inequality, and the fact that $\xi$ is sufficiently small compared to $\eps$ in the last inequality.
    Thus, by the union bound,
         \begin{multline}\label{eqn:star}
            \Pr(\exists u\in \SMA(G_{m^-}): \text{there is a cycle of length at most $4$ containing $u$ in $G_{m^+}$}) \\
            \leq n\cdot \binom{n-1}{2} (n-3)\cdot (2\log n/n)^4 \cdot n^{-1/2-3\eps/4} = o(1)\,.
        \end{multline}

    For (4), it suffices to show that for any distinct $u,v\in \SMA(G_{m^-})$, we have $\dist_{G_{m^+}}(u,v)>4$. Fix arbitrary $u,v\in V(G)$. The number of paths of length $\ell \leq 4$ between them is at most $\binom{n}{\ell-1}\leq n^{\ell-1}$. The probability that a fixed such path exists in $G_{m^+}$ is at most 
    \[
        \frac{\binom{e(G)-\ell}{m^+ -\ell}}{\binom{e(G)}{m^+}} \leq (2\log n/n)^\ell.
    \]
    Let $d_v = \deg_G(v)$ and $d_u = \deg_G(u)$. 
    Conditioned on the existence of one such path in $G_{m^+}$, the probability that both $u,v\in \SMA(G_{m^-})$ is at most 
    \[
        \sum_{i=0}^{\xi\log n}\sum_{j=0}^{\xi\log n} \binom{d_v}{i}\binom{d_u}{i} \frac{\binom{e(G)-d_u-d_v+3}{m^- - i-j}}{\binom{e(G)}{m^-}}\leq n^{-1-3\eps/2}\,
    \]
    by a computation similar to~\eqref{eqn:computation}.
    Thus, we obtain
    \begin{align*}
            \Pr(\exists u,v\in \SMA(G_{m^-}): \dist_{G_{m^+}}(u,v)\leq 4)\leq \sum_{\ell=1}^4 n^2\cdot n^{\ell-1} (2\log n/n)^{\ell}n^{-1-3\eps/2} = o(1)\,,
        \end{align*}
     as desired.
    \end{proof}

\subsection{Finding a good expander}
Recall the definition of almost $2$-factor from~\Cref{sec:2-factor}.
We now prove that, whp, there exists a sparse expander in $G_\tau$ that has an almost $2$-factor:
\begin{lemma}\label{lem:sparse-expander-almost-2factor-exists}
    Whp, $G_\tau$ contains a connected spanning subgraph $R$ such that $R$ is an $(\alpha n,2)$-expander, $e(R)\leq 4\xi n \log n$ and $R$ contains an almost $2$-factor.
\end{lemma}

We first show that there exists a connected sparse expander:
\begin{lemma}\label{lem:sparse-expander-exists}
    Whp $G_\tau$ contains a connected spanning subgraph $R$ such that $R$ is an $(\alpha n,2)$-expander and $e(R)\leq 3 \xi n \log n$.
\end{lemma}
\begin{proof}
    Recall that $\SMA(H) = \{v\in V(H): d_{H}(v)\leq \xi \log n\}$. 
    Let $G_1\subseteq G_\tau$ be obtained as follows: For each vertex $v\in V(G)\setminus \SMA(G_\tau)$, choose a uniformly random set of $\xi\log n$ incident edges. 
    We first analyse the expansion properties of $G_1$:
    \begin{claim}\label{claim:4-expander}
        Whp, for every $S\subseteq V(G)\setminus \SMA(G_\tau)$ such that $S\subseteq A$ or $S\subseteq B$ and $|S|\leq \alpha n$, we have $|N_{G_1}(S)|\geq 4|S|$.
    \end{claim}
    \begin{proofclaim}
        Suppose for contradiction that there exists a set $S\subseteq A\setminus \SMA(G_\tau)$ of size at most $\alpha n$ such that its neighbourhood $T = N_{G_1}(S)\subseteq B$ has size at most $4|S|$ (the symmetric case where $S\subseteq B$ and $T\subseteq A$ can be handled similarly).
        We distinguish cases according to $|S|$. 
        \begin{itemize}
            \item If $|S|\leq \xi \log n/4$, note that for any $v\in S$, we have $|T|\geq d_{G_1}(v) \geq \xi \log n\geq 4|S|$. 
            \item If $\xi\log n/4\leq |S|\leq n/\sqrt{\log n}$, note that $e(G_\tau[S\cup T]) \geq e(G_1[S\cup T])\geq \xi|S|\log n$ and $|S\cup T|\leq 5|S|$. 
            We now bound the probability that such a set $S\cup T$ exists in $G_\tau$.
            Using~\Cref{cor:monotonicity}, we do the computation in $G^+$.
            By~\Cref{prop:basic-inequality} and the union bound, the probability that there exists such a set $S\cup T$ is at most
            \begin{multline*}
                \sum_{i= 5\xi\log n/4}^{5n/\sqrt{\log n}}\binom{n}{i}\cdot \Pr\left(\Bin\left(\binom{i}{2},p_+\right) \geq \xi i\log n/5\right)
                \leq \sum_{i= 5\xi\log n/4}^{5n/\sqrt{\log n}}\left(\frac{en}{i}\right)^i \binom{i^2/2}{\xi i\log n/5} p_+^{\xi i\log n/5} = o(1)\,.
            \end{multline*}
            \item If $n/\sqrt{\log n}\leq |S|\leq \alpha n$, then there is no edge between $S$ and $B\setminus T$ in $G_1$. 
            We first bound the probability that there exists a set $S$ of this size such that $e_{G_\tau}(S,B\setminus T)< n\sqrt{\log n}/20$.
            By~\Cref{cor:monotonicity}, we do the computation in $G^-$.
            
           Note that there are at least $(1/2 + \eps - 4\alpha)n\cdot n/\sqrt{\log n} = (1/10+\eps)n^2/\sqrt{\log n}\geq n^2/(10\sqrt{\log n})$ edges between $S$ and $B\setminus T$ in $G$.
            Each edge is included in $G^-$ independently with probability $\log n/n$. Thus, by Chernoff's bound and the union bound over all $S$ and $B\setminus T$, we have $e_{G^-}(S,B\setminus T)\geq n\sqrt{\log n}/20$ whp. 
            
            Let $F = E_{G_\tau}(S,B\setminus T)$, and for each vertex $v\in S$, let $q_v$ denote the number of edges of $F$ incident to $v$, so that $\sum_{v\in S}q_v = |F|\geq n\sqrt{\log n}/20$ whp.
            Now we bound the probability that none of $F$ appears in $G_1$. 
            Since $\Delta(G_\tau)\leq 20\log n$ whp by~\Cref{lem:small-property}, 
            we deduce that the probability that none of the $\xi\log n$ edges selected by $v$ contain an edge from $F$ is at most 
            \[
                \frac{\binom{\deg_{G_\tau}(v)-q_v}{\xi\log n}}{\binom{\deg_{G_\tau}(v)}{\xi\log n}}\leq \exp(-\xi\log n\cdot q_v/\deg_{G_\tau}(v))\leq \exp(-\xi q_v/20)\,.
            \]
            Note that the edge sets selected by distinct vertices $v\in S\subseteq A$ are disjoint, so the events concerning these vertices are independent. Thus, 
            \[
                \Pr(E_{G_1}(S,B\setminus T) = \emptyset)\leq \exp\left(-\frac{\xi}{20}\sum_{v\in S}q_v\right) = \exp(-\Omega(n\sqrt{\log n}))\,.
            \]
            The union bound over all $S$ and $T$ finishes the proof. 
        \end{itemize}
        Thus, the statement holds in all cases.
    \end{proofclaim}
    Let $G_2\subseteq G_\tau$ be the graph obtained by taking all edges incident to $\SMA(G_\tau)$ and let $H = G_1\cup G_2$. 
    We now show the expansion property of $H$:
    \begin{claim}\label{cla:G-expander}
        Whp $H$ is an $(\alpha n,2)$-expander.
    \end{claim}
    \begin{proofclaim}
        Suppose, for contradiction, that there exists a set $S\subseteq A$ of size at most $\alpha n$ such that its neighbourhood $T = N_{H}(S)\subseteq B$ has size at most $2|S|$. 
        Let $S_1 = S \cap \SMA(G_\tau)$ and $S_2 = S\setminus \SMA(G_\tau)$. 
        Note that $|N_{H}(S_1)| \geq 2|S_1|$. 
        Indeed, at time $\tau$, every $u\in S_1$ has at least two $H$-neighbours. If $|N_H(S_1)|<2|S_1|$, this would contradict~\Cref{lem:small-property}(4) by giving length-$2$ paths between two vertices of $\SMA(G_\tau)$.
        By~\Cref{claim:4-expander}, we know $|N_{H}(S_2)|\geq 4|S_2|$. By~\Cref{lem:small-property}, for any $v\in S_2$, we know that $|N_{H}(S_1)\cap N_{H}(v)|\leq 1$.
        Indeed, by~\Cref{lem:small-property}(3), for each $u \in S_1$, we have $|N_H(u)\cap N_H(v)|\leq 1$; otherwise, we obtain a $4$-cycle containing a $\SMA(G_\tau)$ vertex. Moreover, suppose there are distinct $u,u'\in S_1$ such that $|N_H(u)\cap N_H(v)| = |N_H(u')\cap N_H(v)| =  1$. Then there exists a path of length $4$ between $u$ and $u'$, contradicting~\Cref{lem:small-property}(4).
        Thus, $$|N_H(S)| = |N_H(S_1)| + |N_H(S_2)| - |N_H(S_1)\cap N_H(S_2)| \geq 2|S_1| + 4|S_2| - |S_2| \geq 2|S|\,,$$
        as desired.
    \end{proofclaim}
    From~\Cref{cla:G-expander}, it is immediate that every component of $H$ has size at least $4\alpha n$, with at least $2\alpha n$ vertices on each side.
    Thus, the total number of components of $H$ is at most $1/(2\alpha) = O(1)$.
    We now show that there is a small set of edges that ``merge'' all components of $H$:
    \begin{claim}
        Whp, there exists a set of edges $T\subseteq E(G_\tau)$ of size at most $1/(2\alpha)$ such that $H\cup T$ is connected.
    \end{claim}
    \begin{proofclaim}
        Let $\ell\leq 1/(2\alpha)$ be the number of components of $H$ and let $C_1,\dots, C_\ell$ be an enumeration of the components. Moreover, let $A_i = C_i\cap A$ and $B_i = C_i\cap B$ for each $i\in [\ell]$. Using the $(\alpha n,2)$-expansion, we know $|A_i|,|B_i|\geq 2\alpha n$ for all $i\in [\ell]$.

        Define an auxiliary component graph $G_{\mathrm{comp}}$ as follows: the vertices of $G_{\mathrm{comp}}$ are $\{C_1,\dots, C_\ell\}$, and we put an edge between components $i,j\in V(G_{\mathrm{comp}})$ if there exists an edge between $C_i$ and $C_j$ in $G_\tau$. The goal is to show that $G_{\mathrm{comp}}$ is connected, so that we can find a spanning tree in $G_{\mathrm{comp}}$. Replacing each edge of the spanning tree with an edge in $G_\tau$ that connects the corresponding two components gives the desired $T$.

        Suppose for contradiction that $G_{\mathrm{comp}}$ is not connected. Take a non-empty proper union of components and write $X\subseteq A$ and $Y\subseteq B$ for its two sides. 
        Since both the union and its complement contain whole components, we know $|X|, |A\setminus X|, |Y|, |B\setminus Y|\geq 2\alpha n$. If there is no $G_\tau$ edge crossing the component cut, then 
        \begin{equation}\label{eqn:no-edge}
            e_{G_\tau}(X, B\setminus Y) + e_{G_\tau}(A\setminus X, Y) = 0\,.
        \end{equation}
        However, using $\delta(G) \geq (1/2 + \eps)n$, we conclude 
        \[
            e_G(X,B\setminus Y) \geq |X|\cdot \max\{(1/2+\eps)n - |Y|, 0\} \text{ and } e_G(A\setminus X, Y) \geq |A\setminus X|\cdot \max\{|Y| - (1/2-\eps)n, 0\}\,.
        \]
        By a case distinction on $|Y|$, we have $e_G(X,B\setminus Y) + e_G(A\setminus X, Y)= \Omega(n^2)$. 
        We claim there is no such $X$ and $Y$ satisfying~\eqref{eqn:no-edge}. Indeed, by~\Cref{cor:monotonicity}, we may do the computation in $G^-$ and then
        \[
            \Pr(e_{G^-}(X, B\setminus Y) + e_{G^-}(A\setminus X, Y) = 0)\leq \Pr(\Bin(\Omega(n^2), \log n/n) = 0) = o(4^{-n})\,.
        \]
        Applying the union bound over all choices of $X$ and $Y$ proves that, whp, there is at least one edge across each component cut, and so $G_{\mathrm{comp}}$ is connected.
    \end{proofclaim}    
  Let $R = H\cup T$. It remains to show that $R$ is the desired graph. Note that $R$ is connected and spanning by definition and is an $(\alpha n,2)$-expander by~\Cref{cla:G-expander}.
  Also, note that 
  $$e(R)\leq e(G_1) + e(G_2) + e(T) \leq 2\xi n\log n + \xi \log n\cdot n^{(1-\eps)/2} + O(1)\leq 3\xi n\log n\,,$$
  as desired.
\end{proof}

Next, we show that, whp, $G_\tau$ contains a $2$-factor. 
\begin{lemma}\label{lem:2-factor-exists}
    Whp, $G_\tau$ contains a $2$-factor.    
\end{lemma}
\begin{proof}
    We first establish an easy fact: whp, in $G_\tau$, there are many edges between two sufficiently large sets.
    \begin{claim}\label{cla:many-edges}
        Let $S\subseteq A$ and $T\subseteq B$ with $|S|\geq \alpha n$, $|T|\geq \alpha n$ and $|S|+|T| > n$. Then, whp, $e_{G_\tau}(S,T)\geq 2n$. 
    \end{claim}
    \begin{proofclaim}
        Since $|S|+|T|>n$, one of $S$ or $T$ has size at least $n/2$. Without loss of generality, we assume $|T|\geq n/2$. Since $\delta(G) \geq (1/2 + \eps)n$, we know that $e_G(S, T)\geq \alpha \eps n^2$. 
        By~\Cref{cor:monotonicity}, it suffices to do the computation in $G^-$.
        By Chernoff's bound,
        \[
            \Pr(e_{G^-}(S,T) < 2n)\leq \Pr(\Bin(\alpha\eps n^2, \log n/n) \leq 2n) = \exp(-\Omega_\eps(n\log n))\,.
        \]
        The union bound over all choices of $S$ and $T$, of which there are at most $4^n$, finishes the proof.
    \end{proofclaim}
    Fix an arbitrary set $S \subseteq A$ (the case where $S\subseteq B$ is completely symmetric) and let $s = |S|$.
    Let $T_b = \Gamma^{\geq 2}_{G_\tau}(S)$ and $T_s = \Gamma^{=1}_{G_\tau}(S)$ and write $t_b = |T_b|$ and $t_s = |T_s|$. Our goal is to show that $t_b + t_s/2 \geq s$ and finish the proof using~\Cref{prop:contain-2-factor}.
    
    For the sake of contradiction, assume $t_b + t_s/2 < s$. Let $T' = B\setminus T_b$. Every vertex in $T'$ has at most one neighbour in $S$ in $G_\tau$, and hence $e_{G_\tau}(S,T') = t_s$. Define $r = s-t_b$. Then, by assumption, $e_{G_\tau}(S,T') = t_s < 2r$. We distinguish the following cases:
    \begin{enumerate}
        \item If $|S|\leq \alpha n$, then since $G_\tau$ is an $(\alpha n,2)$-expander (which follows from~\Cref{lem:sparse-expander-exists}), we know $|N_{G_\tau}(S)|\geq 2s$. Note that $N_{G_\tau}(S)\subseteq T_b\cup T_s$. Hence $t_b + t_s/2\geq |T_b\cup T_s|/2 \geq |N_{G_\tau}(S)|/2 \geq s$, a contradiction.
        \item If $|T'|\leq \alpha n$, then again $|N_{G_\tau}(T')|\geq 2|T'|$ by $(\alpha n,2)$-expander. At most $n - s$ vertices of $N_{G_\tau}(T')$ lie outside $S$, so $|N_{G_\tau}(T')\cap S|\geq 2|T'| - (n-s) = |T'| + r$. Since $r\leq |T'|$, we get $|N_{G_\tau}(T')\cap S|\geq 2r$ and thus $e_{G_\tau}(S,T')\geq 2r$, a contradiction.
        \item If $|S|\geq \alpha n$ and $|T'|\geq \alpha n$, then by~\Cref{cla:many-edges}, we have $e_{G_\tau}(S, T')\geq 2n\geq 2r$, again a contradiction.
    \end{enumerate}
    Thus, in all cases, we have reached a contradiction. Hence $t_b+t_s/2\geq s$, and therefore $G_\tau$ has a $2$-factor.
\end{proof}

We can now deduce~\Cref{lem:sparse-expander-almost-2factor-exists} immediately:
\begin{proof}[Proof of~\Cref{lem:sparse-expander-almost-2factor-exists}]
    From~\Cref{lem:sparse-expander-exists}, there exists a connected, spanning, $(\alpha n,2)$-expander $R_1$ in $G_\tau$ such that $e(R_1)\leq 3\xi n\log n$. By~\Cref{lem:2-factor-exists}, there exists a $2$-factor $R_2$ in $G_\tau$. Set $R = R_1\cup R_2$. It is clear that $R$ satisfies all stated properties and $e(R)\leq 3\xi n\log n + 2n \leq 4\xi n\log n$. 
\end{proof}

\subsection{Rotation-extension and almost 2-factors}\label{subsec:rotation-extension}
Recall the classical P\'{o}sa rotation technique:
\begin{definition}
    Let $G$ be a graph and $P = (v_1,\dots, v_\ell)$ be a path in $G$. Suppose there is an edge between $v_\ell$ and $v_i$ in $G$ for some $i\in (1,\ell)$. Then, we can obtain a new path $P'$ from $v_1$ to $v_{i+1}$ of the same length as $P$ by removing the edge $\{v_i, v_{i+1}\}$ and adding the edge $\{v_i,v_\ell\}$. Each such operation is a (P\'{o}sa) rotation.
\end{definition}
The aim of P\'{o}sa rotations is to find many pairs of vertices for which there is a path of the same length. If there is an edge connecting two endpoints of such a path, or one endpoint and a vertex outside the path, then connectedness can be used to find a longer path. This latter step is called an ``extension''.

We call a path $P\subseteq G$ \emph{balanced} if $|P\cap A| = |P\cap B|$. In particular, balanced paths have odd length. We call a path $P\subseteq G$ \emph{strongly maximal} if there is no cycle on $V(P)$ and $P$ cannot be extended in $G$ using P\'{o}sa rotations. 

The following lemma is a standard P\'{o}sa rotation lemma~\cite{Posa} adapted to our setting. For completeness, we include the proof.

\begin{lemma}\label{lem:bipartite-posa-rotation}
    Let $R\subseteq G$ be a connected, non-Hamiltonian $(\alpha n,2)$-expander. Let $P$ be a balanced strongly maximal path in $R$ and $t = \alpha n$. There exist distinct vertices $x_1,\dots, x_t\in A$ and sets $Y_1,\dots, Y_t\subseteq B$ such that $|Y_i|\geq \alpha n$ and $R$ has no $x_i -Y_i$ edges for every $i\in [t]$. Moreover, for all $i\in [t]$ and all $y_i\in Y_i$, the graph $R$ has a path $Q$ from $x_i$ to $y_i$ such that $V(Q) = V(P)$ and $Q$ is also strongly maximal.  
    The same holds with the roles of $A$ and $B$ swapped.
\end{lemma}
\begin{proof}
    Let $P = (v_0,v_1,\dots, v_\ell)$. Without loss of generality, $v_0\in B$ and $v_\ell \in A$.
    Let $X\subseteq A$ denote the set of all endpoints that can be obtained from $P$ via P\'osa rotations while keeping $v_0$ fixed.
    Define 
    $$X' = \{v_{j-1}: v_j\in X, \,j\geq 1\}\cup \{v_{j+1}: v_j\in X, \,j\leq \ell-1\} \subseteq B\,.$$
    Note that $|X'|\leq 2|X|-1$ because $v_\ell\in X$ and contributes only one vertex to $X'$. 

    We first observe the following simple invariant:
    \begin{claim}\label{cla:invariant}
        Let $Q$ be a path from $v_0$ to some vertex in $X$ that is obtained from $P$ by P\'osa rotations while keeping $v_0$ fixed. 
        Then every edge $e\in E(Q)\setminus E(P)$ has its $B$-endpoint in $X'$.
    \end{claim}
    \begin{proof}
          We prove this by induction on the number of rotations used to obtain $Q$. The base case, where no rotations are used, is trivial by the definition of $X'$. 
        Suppose a rotation is performed on a current path $Q$ with endpoint $x\in A$, using an edge $xy$, where $y\in B$ is the pivot. The rotation deletes the path edge $yz$, where $z\in A$ is the neighbour of $y$ lying towards $x$ on $Q$ so that $z$ becomes the new endpoint. If $\{y,z\}\in E(P)$, then $y$ is an original neighbour of $z\in X$ so $y\in X'$. If $\{y,z\}\notin E(P)$, then by induction hypothesis, $y$ already lies in $X'$. Thus, the newly added edge $\{x,y\}$ has its $B$-endpoint $y\in X'$.
    \end{proof}
    Next, we prove that every neighbour of $X$ in $R$ must be in $X'$:
    \begin{claim}
        $N_R(X)\subseteq X'$.
    \end{claim}
    \begin{proof}
        Fix an arbitrary $y\in N_R(X)$ and an edge $\{x,y\}\in E(R)$ where $x\in X$.
        The goal is to show $y\in X'$. By the definition of $X$, there exists a path $Q$ on $V(P)$ with endpoints $v_0$ and $x$ that is obtained from $P$ by P\'osa rotations. 
        Since every path obtained from $P$ by P\'osa rotations is strongly maximal, we know $Q$ is strongly maximal.
        We may assume $y\in V(Q)$ and $y\neq v_0$. Indeed, in the former case, we have found a longer path by extending $Q$ using $\{x,y\}$. In the latter case, $Q+\{x,y\}$ gives a cycle on $V(P)$. Either way, the strong maximality of $P$ is violated.
         
        If $\{x,y\}\in E(Q)$, then it is the last edge of $Q$. If $\{x,y\}\in E(P)$, then $y\in X'$. If $\{x,y\}\notin E(P)$, then by~\Cref{cla:invariant}, $y\in X'$.
        If $\{x,y\}\notin E(Q)$, then note that $z\in X$, where $z\in A$ is the neighbour of $y$ lying towards $x$ on $Q$. If $\{y,z\}\in E(P)$, then $y\in X'$. If $\{y,z\}\notin E(P)$, then by~\Cref{cla:invariant}, $y\in X'$.
    \end{proof}
    We may assume $|X|>\alpha n$; otherwise, by the definition of the $(\alpha n,2)$-expander, we know $|N_R(X)|\geq 2|X|$, a contradiction.
    Recall that $t = \alpha n$.
    Choose distinct vertices $x_1,\dots, x_t\in X\subseteq A$. For each $i\in [t]$, there is a path $P_i$ on $V(P)$ with endpoints $v_0$ and $x_i$, obtained from P\'osa rotations while fixing $v_0$. 
    Now repeat the same argument for each such path, but with $x_i$ fixed. For each $i\in [t]$, we obtain a set of vertices $Y_i\subseteq B$ of size $\alpha n$. For each $y_i\in Y_i$, there is a path $Q\subseteq R$ from $x_i$ to $y_i$ with $V(Q) = V(P)$. Note that $Q$ is strongly maximal; otherwise, it contradicts the strong maximality of $P$. This implies that there are no $x_i-Y_i$ edges. 
    
    The argument is symmetric in $A$ and $B$, so the same proof applies with $A$ and $B$ swapped.
\end{proof}

For an almost $2$-factor $H$, we introduce the following potential function:
\[
    \ell(H) = \begin{cases}
        e(P) \quad&\text{ if $H$ contains a component $P$ that is a path}\\
        \max\{e(Q): \text{$Q$ is a cycle in $H$}\}\quad&\text{otherwise}
    \end{cases}
\]
For a graph $G$, define 
$$L(G) = \max\{\ell(H): \text{$H$ is an almost $2$-factor of $G$}\}\,.$$
If $G$ does not contain an almost $2$-factor, then $L(G) = 0$. Also, if $L(G) = 2n$, then $G$ is Hamiltonian.

We record a simple observation:
\begin{proposition}\label{prop:longest-path}
    Let $G$ be a connected balanced bipartite graph with $L(G) > 0$. Let $H$ be an almost $2$-factor that maximises $\ell(H)$ (i.e.\ $\ell(H) = L(G)$). If $H$ is not a Hamiltonian cycle, then it contains a balanced strongly maximal path.
\end{proposition}
\begin{proof}
    Assume for contradiction that every component of $H$ is a cycle. 
    Because $H$ is not a Hamilton cycle, $H$ has at least two components.
    Let $Q^\star$ be a longest cycle in $H$. Because $G$ is connected, there must be an edge $e = (u,v)$ connecting $u\in V(Q^\star)$ to $v\in Q'$, where $Q'$ is a different cycle.
     Both cycles have length at least $4$ since $G$ is bipartite.
     By joining these two cycles with $e$ and removing one edge incident to $u$ in $Q^\star$ and another incident to $v$ in $Q'$, we obtain a path $P$ such that $e(P) > e(Q^\star)$, a contradiction. 
     
     Since every other component of $H$ is an even cycle and $G$ is a balanced bipartite graph, the path must be balanced. 
     
     The path must also be strongly maximal. 
     Indeed, if there is a cycle $C$ on $V(P)$, then replacing $P$ with $C$ increases $\ell$ by $1$, which contradicts the maximality of $H$. 
     If $P$ can be extended using P\'{o}sa rotations, let $v$ be the new vertex that was not on $P$. Let $C$ denote the component of $H$ containing $v$, which must be an even cycle. 
     Delete an arbitrary edge $e$ incident to $v$ in $C$ and concatenate $C\setminus \{e\}$ with the rotated path on $V(P)$ whose endpoint has an edge to $v$. The resulting graph is still an almost $2$-factor. However, the length of the path has strictly increased, contradicting the maximality of $H$.
\end{proof}

Now we introduce the notion of a booster pair:
\begin{definition}
    Let $R$ be a spanning subgraph of $G$. We say $\{e_1,e_2\}\subseteq E(G)\setminus E(R)$ is a \emph{booster pair} for $R$ if $R' = R\cup \{e_1,e_2\}$ is either Hamiltonian or satisfies $L(R')> L(R)$.
\end{definition}

Also, for convenience, we say that an $(\alpha n,2)$-expander $R\subseteq G$ is \emph{good} if it is connected, is non-Hamiltonian, satisfies $L(R)>0$ (i.e.\ contains an almost $2$-factor), and has $e(R)\leq 5\xi n\log n$.
We now state the main lemma of this subsection:
\begin{lemma}\label{lem:existence-booster-pair}
    Whp, every good $(\alpha n,2)$-expander $R\subseteq G_\tau$ has a booster pair in $G_\tau\setminus R$.
\end{lemma}
The key part of this lemma is proving that there are many booster pairs for $R$ and that they are relatively ``well-spread''.
\begin{lemma}\label{lem:many-boosters}
    For any $\eps\in (0,1)$, there is a constant $C=\Omega(\eps)$ such that the following holds. 
    Let $R\subseteq G$ be a good $(\alpha n,2)$-expander. Then there exists a family $\calB_R$ of booster pairs for $R$ such that $|\calB_R|\geq Cn^3$ and every $e\in E(G)$ is contained in at most $2n$ pairs of $\calB_R$.
\end{lemma}
\begin{proof}
    Since $R$ is a good $(\alpha n,2)$-expander, there exists an almost $2$-factor. Let $P$ be a balanced strongly maximal path in an almost $2$-factor attaining $L(R)$; such a path exists by~\Cref{prop:longest-path}. 
    By~\Cref{lem:bipartite-posa-rotation}, there is a set of at least $\alpha n$ vertices $S$ such that for each $s\in S$, there is a set of at least $\alpha n$ vertices $T_s$ such that for each $t\in T_s$, there is a path $P_{s,t}\subseteq R$ of the same length as $P$ with vertex set $V(P)$. 
    Without loss of generality, assume $S\subseteq A$.

    Let $\eta = \eps/100$. 
    Let $A' = A\setminus V(P)$ and $B' = B\setminus V(P)$. 
    We call an endpoint $s\in S\cap A$ \emph{good} if $|N_{G}(s)\cap B'|\leq \eta n$ and \emph{bad} otherwise. Also, we call an endpoint $t\in (\cup_{s\in S}T_s)\cap B$ \emph{good} if $|N_{G}(t)\cap A'|\leq \eta n$ and \emph{bad} otherwise.

    Suppose that for at least $\alpha^2n^2/2$ pairs $(s,t)$, at least one endpoint is bad. Since every vertex appears in at most $n$ pairs, there are at least $\alpha^2 n/2$ bad vertices.
    Without loss of generality, assume that at least $\alpha^2 n/4$ vertices $s\in A$ are bad.
    
    Observe that if $s$ is bad, then every neighbour of $s$ in $B'$ strictly increases $L(R)$. 
    Indeed, replacing the path component of $P$ by the rotated path with endpoint $s$ and then using the edge from $s$ to a vertex $s'$ of a cycle component outside $V(P)$ while breaking an edge incident to $s'$ on the cycle, we obtain a longer path than $P$.
    Thus, such an edge is not in $R$; otherwise, it would contradict strong maximality.
    For each such edge $e$, we can pair it with any edge $f\in E(G)\setminus (E(R)\cup \{e\})$ to get a booster pair (note that $f$ is only for padding since $e$ itself already increases $L(R)$).
    
    Now we show how to obtain $\Omega(n^3)$ booster pairs while ensuring that each edge appears at most $2n$ times. 
    Since each bad vertex $s\in A$ is bad, it has at least $\eta n$ incident edges to vertices outside the path. 
    For concreteness, let us fix a subset of bad vertices $D\subseteq A$ of size $\alpha^2 n/4$ and a set of $\eta n$ edges incident to each vertex.
    Let $F$ denote the set of all such edges. Note that $|F|\geq \alpha^2 \eta n^2/8$.
    We now pair each such edge with $2n$ edges in $E(G)\setminus E(R)$ in such a way that no edge appears in more than $2n$ pairs. 
    Note that this gives the desired $\calB_R$.
    Consider all edges in $E(G)\setminus E(R)$ that are not incident to any vertex of $D$. The number of such edges is at least $(1/2+\eps )n^2 - 5\xi n\log n - \alpha^2 n^2/2 \geq 0.49n^2$ for all sufficiently large $n$. We may order them arbitrarily and then assign to each edge $e\in F$ its $2n$ candidates cyclically.
    Edges in $F$ appear exactly $2n$ times by construction. Edges not incident to $D$ appear at most $2n|F|/0.49n^2 \leq 2n$ 
    times, since $|F|\leq |D|\eta n = \alpha^2\eta n^2/4$ and $\eta\leq 1/100$. This gives at least $2n|F|= \Omega(n^3)$ booster pairs.

    Otherwise, for at least $\alpha^2n^2/2$ pairs $(s,t)$, both endpoints are good. In this case, for each pair $(s,t)$, let $P_{s,t} = (v_1,\dots, v_\ell)$ and  
    define $I_{s,t} = \{1\leq i<\ell: \{s,v_{i+1}\}, \{t,v_i\}\in E(G)\}$. 
    Note that every $i\in I_{s,t}$ gives a booster pair $\{sv_{i+1}, tv_{i}\}$, since adding these two edges creates a cycle on $V(P)$ and replacing the path component of the maximizing almost $2$-factor by this cycle increases $\ell$ by 1.
    Since both endpoints are good, we know that 
    $$|I_{s,t}|\geq \deg_G(s, V(P)\cap B ) + \deg_G(t, V(P)\cap A) - n - O(1)\geq 2(\eps -\eta)n - O(1)\,,$$
    where the $O(1)$ comes from the possible loss near the end of the path.
    
    Observe that each edge appears at most $2n$ times. Indeed, if an edge appears as $\{s,v_{i+1}\}$, then one of the endpoints is $s$ and there are at most $n$ choices for $t$. Similarly, if it appears as $\{t,v_i\}$, then the same reasoning gives at most $n$ appearances. 
    To count booster pairs, note that we have $\alpha^2 (\eps - \eta)n^3 = \Omega(\eps n^3)$ candidates by taking the union of $I_{s,t}$ over all such pairs $(s,t)$. The number of such pairs that intersect $R$ is at most $O(n\cdot e(R)) = o(n^3)$. 
    Also, every edge pair can participate at most $4! = O(1)$ times in $\cup I_{s,t}$, where the union is over the set of $(s,t)$ pairs defined above (by specifying which vertex plays the role of $s$, $t$, $v_i$ and $v_{i+1}$). 
    Thus, in both cases we have at least $\Omega(n^3)$ booster pairs in total. 
\end{proof}

We now prove~\Cref{lem:existence-booster-pair}:
\begin{proof}[Proof of~\Cref{lem:existence-booster-pair}]
    Recall that $m^- = e(G)\cdot p_-$ and $m^+ = e(G)\cdot p_+$. From~\Cref{lem:hitting-time-window}, we know whp $m^-\leq \tau\leq m^+$. 
    Fix an arbitrary $m \in [m^-,m^+]$ and an arbitrary good $(\alpha n,2)$-expander $R\subseteq G$. Let $r = e(R)$. Note that $\Pr(R\subseteq G_{m}) = \binom{e(G)-r}{m-r}/\binom{e(G)}{m}\leq (p_+)^r$. Let $m' = m - r$. Observe that conditioned on $R\subseteq G_m$, the remaining graph $G' = G_m\setminus R$ is distributed as a uniformly random set of $m'$ edges in $E(G)\setminus E(R)$.  
    
    By~\Cref{lem:many-boosters}, we know there is a family $\calB_R$ of at least $Cn^3$ booster pairs and each edge of $G$ is involved in at most $2n$ of them. Build an auxiliary graph $\Gamma$ whose vertices are $E(G)\setminus E(R)$ and there is an edge between two vertices $u,v\in V(\Gamma)$ if and only if $(u,v)\in \calB_R$. Note that $\Delta(\Gamma)\leq 2n$ and $e(\Gamma)\geq Cn^3$. 
   For concreteness, assume $e(\Gamma)= Cn^3$.
    It remains to show that a uniformly chosen $m'$-vertex subset of $V(\Gamma)$ induces at least one edge. To do so, we invoke~\Cref{prop:janson} with the following parameters: 
    \[
        \Omega \leftarrow E(G)\setminus E(R),\; n\leftarrow e(G) - r,\; m\leftarrow m',\; I = \calB_R\,.
    \]
    Note that $p = (m-r)/(e(G)-r)\geq 0.99p_-$. Also, 
    \[
        \mu = \sum_{i\in I}p^2 = Cn^3 p^2\quad\text{ and }\quad \Delta = \sum_{\substack{(e,f)\in \calB_R^2:\\ e\neq f, \\e\cap f\neq \emptyset}}p^{3} \leq  2\cdot Cn^3 \cdot 2n\cdot p^3 = 4Cn^4 p^3\,.
    \]
    Applying~\Cref{prop:janson} with $q = \mu/(2\Delta)<1$, we conclude that, conditioned on $R\subseteq G_m$, with probability at least $1-\exp(-\Omega(Cn\log n))$, there exists at least one booster pair in $G'$.
    Thus, by the union bound, the probability that~\Cref{lem:existence-booster-pair} fails is at most 
    \begin{align*}
        &\sum_{m=m^-}^{m^+}\sum_{\substack{R\subseteq G:\\ \text{$R$ is a good}\\\text{$(\alpha n,2)$-expander}}} \Pr(R\subseteq G_m)\cdot \Pr(\text{there is no booster pair for $R$ in $G_m\setminus R$}\mid R\subseteq G_m)\\
        &\leq  \sum_{m=m^-}^{m^+}\sum_{r=0}^{5\xi n\log n}\binom{e(G)}{r}(2\log n/n)^r\cdot \exp(-\Omega(Cn\log n))\\
        &= o(1)\,,
    \end{align*}
    where we used that $\xi$ is sufficiently small compared to $\eps$ and that $C=\Omega(\eps)$.
\end{proof}

\subsection{Proof of Theorem~\ref{thm:main} and Corollary~\ref{cor:threshold}}
We are now ready to prove~\Cref{thm:main}:
\begin{proof}[Proof of~\Cref{thm:main}]
    By~\Cref{lem:sparse-expander-almost-2factor-exists}, whp there exists a connected spanning subgraph $R\subseteq G_\tau$ such that $R$ is an $(\alpha n,2)$-expander, $e(R)\leq 4\xi n\log n$ and $R$ contains an almost $2$-factor (equivalently $L(R) > 0$). 
    By~\Cref{lem:existence-booster-pair}, there exists a booster pair for $R$. 
    Add this booster pair to $R$. Repeat this step of finding and adding booster pairs.
    Note that there are at most $2n$ steps, since if $L(R) = 2n$, then $R$ must be Hamiltonian. Moreover, because the total number of edges added is at most $4n$, all properties of good expanders are preserved under addition of edges, and~\Cref{lem:existence-booster-pair} applies to every good $(\alpha n,2)$-expander. Thus the procedure can continue, since $4\xi n\log n + 4n< 5\xi n\log n$.
\end{proof}

The proof of~\Cref{cor:threshold} is essentially a first and second moment computation similar to the proof of~\Cref{lem:hitting-time-window}. 
\begin{proof}[Proof of~\Cref{cor:threshold}]
    Let $p = (\log n + \log\log n +\gamma)/\delta(G)$, where $\gamma = \gamma(n) = \omega(1)$.
    By~\Cref{prop:Gm-to-Gp}, it suffices to prove that $G_{m'}$ is Hamiltonian whp for every $m' = e(G)p\pm\lambda\sqrt{e(G)p(1-p)}$. where $\lambda = \omega(1)$ and $\lambda \sqrt{\log n/n} = o(\gamma)$. 

    Let 
    $$p' = m'/e(G) \geq p - \lambda\sqrt{p(1-p)/e(G)}\geq (\log n + \log\log n+\gamma/2)/\delta(G)\,.$$
    Then,
    $$\Pr(\deg_{G_{p'}}(v)\leq 1)\leq \Pr(\Bin(\delta(G), p')\leq 1)\leq (1-p')^{\delta(G)} + \delta(G)\cdot p' \cdot (1-p')^{\delta(G)-1} = o(n^{-1})\,.
    $$
    By the union bound over all vertices, we conclude that, whp, there is no vertex of degree at most $1$ in $G_{p'}$. 
    Thus, by~\Cref{prop:model-equiv}, we conclude that, whp, there is no vertex of degree at most $1$ in $G_{m'}$. Hence $m'\geq \tau_2(G)$. By~\Cref{thm:main}, whp $\tau_\calH(G)= \tau_2(G)\leq m'$, and thus $G_{m'}$ is Hamiltonian whp, which finishes the proof.

    Let $p^- = (\log n + \log\log n -\omega(1))/\delta(G)$.
    Let $S$ be the set of vertices with degree $\delta(G)$ on $A$. Without loss of generality, assume $|S|\geq \beta n/2$. 
    If $p\leq p^-$, then for $v\in S$, let $X_v$ be the indicator random variable for $\deg_{G_{p}}(v) \leq 1$ and $X = \sum_{v\in S}X_v$. Thus, 
    \[
        \E[X] = \sum_{v\in S}\Pr(\Bin(\delta(G), p) \leq 1) = |S|\left(
            (1-p)^{\delta(G)} + \delta(G) \cdot p \cdot (1-p)^{\delta(G)-1}
        \right)\geq e^{\omega(1)} = \omega(1)\,.
    \]
    Since the random variables $X_v$ for $v\in S$ depend on disjoint sets of edges, they are independent, and so $\Var[X] \leq \E[X]$. By Chebyshev's inequality, $\Pr(X > 0) = 1-o(1)$. Therefore, $G_{p}$ is non-Hamiltonian whp.
\end{proof}

\section{Concluding remarks}
We have shown that for any $\eps \in (0,1/2]$, in a balanced bipartite graph on $2n$ vertices with minimum degree at least $(1/2 + \eps)n$, whp $\tau_2(G) = \tau_{\calH}(G)$. An immediate natural question is whether one can push the minimum degree all the way to the Moon-Moser bound $(n+1)/2$. The following construction gives a negative answer. 
\begin{proposition}\label{prop:counter-example}
    There exists a balanced bipartite graph on $2n$ vertices such that with positive probability, $\tau_2(G)< \tau_{\calH}(G)$.
\end{proposition}
\begin{proof}[Proof sketch]
    Let $G$ be a balanced bipartite graph on $4m+2$ vertices (so that $n = 2m+1$) with sides $A$ and $B$. Let $A = A_0\cup A_1$ with $|A_0| = m$ and $|A_1| = m+1$ and let $B = B_0\cup B_1$ with $|B_0| = m+1$ and $|B_1| = m$. Add all edges between $A_0$ and $B_0$, between $B_0$ and $A_1$ and between $A_1$ and $B_1$. It is easy to check that the minimum degree of $G$ is $m+1 = (n+1)/2$. 
    
    Let $p = (\log m + \log\log m)/m$ and $\tau = e(G)\cdot p$. 
    Also, let $X$ be the number of vertices of degree at most $1$ in $G_\tau$ and let $Y$ be the number of vertices of $B_0$ that send at most one edge to $A_0$ in $G_\tau$. 
    By a direct computation, $X$ and $Y$ jointly converge to Poisson random variables with parameters $c_1, c_2 =\Omega(1)$, respectively.
    A direct calculation shows that with positive probability, $X = 0$ and $Y\geq 3$. The former implies $\tau_2(G)\leq \tau$. 
    By construction, every Hamilton cycle uses $2$ edges incident to each vertex in $A_0$, which implies that all but at most $2$ vertices in $B_0$ must send at least $2$ edges to $A_0$. Thus, by the latter property, $G_\tau$ is non-Hamiltonian, and so $\tau_{\calH}(G)> \tau \geq \tau_2(G)$ with positive probability.
\end{proof}
This leads us to the following natural question:
\begin{problem}
    Does there exist $g = g(n) = o(n)$ such that the following holds?
    Let $G$ be a regular balanced bipartite graph on $2n$ vertices with degree at least $n/2 + g$. Then whp $\tau_2(G) = \tau_{\calH}(G)$?
\end{problem}
There are very powerful results showing that the structure of regular $\eps$-Dirac bipartite graphs is highly constrained (see e.g.~\cite{Kuhn-Osthus}). 
Another interesting open problem is to investigate whether, under the regularity constraint, the minimum degree condition can be relaxed to $(n+1)/2$:
\begin{problem}
    Let $G$ be a regular balanced bipartite graph on $2n$ vertices with degree at least $(n+1)/2$. Is it then true that $\tau_2(G) = \tau_{\calH}(G)$ whp?
\end{problem}
More broadly, it is of interest to find other conditions under which a balanced bipartite graph satisfies the hitting-time theorem for Hamiltonicity. One particular case, where $G$ is a bipartite $C$-expander (see~\cite{Draganic-Kim-Lee-David-Pavez-Sudakov} for definitions), should follow from the same approach used in~\cite{Chen-Chen-Im-Wang}.

\end{document}